\documentclass[12pt]{article}
\usepackage[utf8]{inputenc}
\usepackage{amsmath}
\usepackage{amssymb}
\usepackage{mathrsfs}
\usepackage{amsthm}
\usepackage{cite}
\usepackage{hyperref}
\usepackage{enumerate}
\usepackage{combelow}
\usepackage[left=0.7in,right=0.7in,top=0.7in,bottom=0.7in]{geometry}
\usepackage{titlesec}
\titleformat*{\section}{\large\bfseries}
\newtheorem{theorem}{Theorem}[section]
\newtheorem{lemma}[theorem]{Lemma}
\newtheorem{corollary}[theorem]{Corollary}
\newtheorem{definition}[theorem]{Definition}
\newtheorem{example}[theorem]{Example}
\newtheorem{proposition}[theorem]{Proposition}

\numberwithin{equation}{section}
\title{Best Proximity Point Results for Cyclic Orbital Contraction Mappings in $CAT_p(0)$ Metric Spaces}
\author{\large  Parveen Kumar$^1$, Ankit Kumar$^2$\footnote{Corresponding author}  and Manu Rohilla$^3$\\ \\ 
{\small $^1,^2$Department of Mathematics, Punjab Engineering College (Deemed to be University)  }\\
{\small Chandigarh-160012, India}\\{\small E-mail: $^1$parvveenbeniwal11@gmail.com, $^2$ankitkumar@pec.edu.in}\\{\small $^3$Department of Mathematics, J.C. Bose University of Science and Technology, YMCA,}\\
{\small Faridabad-121006, India}\\{\small E-mail: manurohilla25994@gmail.com}\\
 }
\date{}
\begin{document}

\maketitle
\begin{abstract}
In this paper, we introduce the concept of cyclic orbital contraction mappings which generalizes the concept of cyclic contraction mappings. We establish the existence of best proximity point of these mappings in the framework of $CAT_p(0)$ metric spaces. Also, we study the existence of best proximity point theorems for cyclic orbital contraction mappings in uniformly convex Banach spaces.
\end{abstract}
 \section{Introduction and preliminaries}
 Fixed point theory provides a framework to solve the equations $\mathfrak{G}\zeta=\zeta$, where $\mathfrak{G}$ is a mapping defined on a subset of a normed linear space or metric space. When $\mathfrak{G}: \Omega \rightarrow \Delta$ is a non-self mapping, a fixed point may fail to exist. In such cases, the objective is to find $\zeta \in \Omega$ which is closest to $\mathfrak{G}\zeta$ in some suitable sense. This leads to the development of best approximation and best proximity point theorems.  According to Fan's \cite{n1} classical approximation theorem, if  $\mathfrak{V}$ is a nonempty, compact and convex subset of  a normed linear space $\mathfrak{N}$ and $\mathfrak{G}:\mathfrak{V} \rightarrow \mathfrak{N}$ is a continuous mapping, then there exists $\varrho^* \in \mathfrak{V}$ such that
$$\Vert \varrho^*-\mathfrak{G}(\varrho^*)\Vert =\min_{\varrho \in \mathfrak{V}} \Vert \varrho-\mathfrak{G}(\varrho^*)\Vert. $$
Reich  \cite{n2}, Prolla \cite{n3}, and Sehgal and Singh \cite{n4,n5} later extended this classical result.  Best proximity point theorems provide optimal approximate solutions. Best proximity points of nonexpansive and contraction  mappings have been widely studied in Banach and metric spaces (see \cite{n6,n7,n8,n9,n10,n11}).

We shall denote the set of natural numbers by $\mathbb{N}$ throughout the paper. Let $\Omega$ and $\Delta$ be nonempty subsets of a metric space $(\mathfrak{X},d)$. Define
\begin{align*}
dist(\Omega,\Delta) & = \inf\{d(\varsigma,\vartheta): \varsigma \in \Omega \thinspace \thinspace \mbox{ and } \thinspace \vartheta \in \Delta \},\\
\Omega_0 &=  \{\varsigma \in \Omega :d(\varsigma,\vartheta) =dist(\Omega,\Delta) \thinspace \thinspace \mbox{for some} \thinspace \vartheta \in \Delta \},\\ 
\Delta_0 & =  \{\vartheta \in \Delta :d(\varsigma,\vartheta ) =dist(\Omega,\Delta) \thinspace \thinspace \mbox{for some} \thinspace \varsigma \in \Omega \}.
  \end{align*}
  Let $\Xi:\Omega\cup \Delta \rightarrow \Omega \cup \Delta$ be a mapping satisfying $\Xi(\Omega) \subset \Delta$ and  $\Xi(\Delta) \subset \Omega$. Then $\varrho \in \Omega \cup \Delta$ is called a best proximity point of $\Xi$ if $d(\varrho,\Xi\varrho)=dist(\Omega,\Delta)$. A best proximity point reduces to a fixed point, when $\Omega \cap \Delta \neq \emptyset $.
  
  Let  $\mathfrak{X}$ be a Banach space. If there exists a strictly increasing function $\Upsilon:(0,2] \rightarrow [0,1]$ such that for all $\sigma_1,\sigma_2,\sigma_3 \in \mathfrak{X}$, $M>0$ and $m \in [0,2M]$ satisfying $\Vert \sigma_1-\sigma_3 \Vert \leq M$, $\Vert \sigma_2-\sigma_3 \Vert \leq M$ and $\Vert \sigma_1-\sigma_2 \Vert \geq m$, we have $\Big\Vert \frac{\sigma_1+\sigma_2}{2}-\sigma_3 \Big\Vert \leq \Big(1-\Upsilon \Big(\frac{m}{M} \Big) \Big)M$. Then  $\mathfrak{X}$ is called a uniformly convex space.

Kirk et al. \cite{n12}  established the nonemptiness of $\Omega_0$ and $\Delta_0$.
\begin{lemma}
\cite[Lemma 3.2]{n12}  Let $\mathfrak{X}$ be a reflexive Banach space. Suppose that $\Omega$ and $\Delta$ are nonempty, closed and convex subsets of $\mathfrak{X}$. Assume that $\Omega$ is bounded. Then $\Omega$ and $\Delta$ are nonempty.
\end{lemma}

Eldred and Veeramani \cite{n6} introduced the notion of cyclic contraction mappings established the existence of best proximity points of these mappings.
\begin{definition}
\emph{\cite[Definition 2.3]{n6} Let $(\mathfrak{X},d)$ be a metric space. Let $\Omega$ and $\Delta$ be nonempty subsets of $\mathfrak{X}$. A mapping $\Xi:\Omega \cup \Delta \rightarrow \Omega \cup \Delta$ is  said to be a cyclic contraction mapping if for all $\varsigma \in \Omega$ and $\vartheta \in \Delta$, we have}

\emph{(i) $\Xi(\Omega) \subset \Delta$ and $\Xi(\Delta) \subset \Omega$,}

\emph{(ii) $d(\Xi\varsigma,\Xi\vartheta) \leq \eta d(\varsigma,\vartheta)+(1-\eta)dist(\Omega,\Delta)$  for some $\eta \in (0,1)$.}
\end{definition}
 \begin{theorem}
\cite[Theorem 3.10]{n6}  Let $\Omega$ and $\Delta$ be nonempty, convex and closed subsets of  a uniformly convex Banach space $\mathfrak{X}$. Let  $\Xi:\Omega \cup \Delta \rightarrow \Omega \cup \Delta$ be a cyclic contraction mapping. Then $\Xi$ has a unique best proximity point $\varrho$ in $\Omega$.
\end{theorem}
 Let $(\mathfrak{X},d)$ be a metric space. For a self mapping $\Xi:\mathfrak{X} \rightarrow \mathfrak{X}$, by convention $\Xi^{n+1}=\Xi \circ \Xi^n$ and $\Xi^0=\mathfrak{I}$, where $\mathfrak{I}$  is the identity mapping on $\mathfrak{X}$. For a given mapping $\Xi:\mathfrak{X} \rightarrow \mathfrak{X}$, the associated orbit and double orbit are defined by
 \begin{align*}
 \mathcal{O}_{\Xi}(\varsigma)&=\{\Xi^n \varsigma:n \in \mathbb{N}\cup \{0\}\},\\
 \mathcal{O}_{\Xi}(\varsigma,\vartheta)&=\mathcal{O}_{\Xi}(\varsigma) \cup \mathcal{O}_{\Xi}(\vartheta).
 \end{align*}
Heged\H us and Szil\'agyi \cite{n13} studied the fixed point of a mapping involving orbits. Walter \cite{n14} explored fixed point results for contraction mapping involving orbits and obtained the following theorem:
\begin{theorem}
\cite[Theorem 4]{n14} Let $(\mathfrak{X},d)$ be a complete metric space. Let $\Xi: \mathfrak{X} \rightarrow \mathfrak{X}$  be a mapping such that all the orbits are bounded and for every $\varsigma, \vartheta  \in \mathfrak{X}$, we have 
$$d(\Xi\varsigma,\Xi\vartheta) \leq  \Gamma(diam \thinspace \mathcal{O}_{\Xi}(\varsigma,\vartheta)),$$
where $\Gamma:[0,\infty) \rightarrow [0,\infty)$ is a function such that $\Gamma$ is continuous, increasing and $\Gamma(w)<w$ for all $w>0$. Then $\Xi$ has a unique fixed point.   
\end{theorem}
In 2016, Bessenyei \cite{n15} rediscovered the theorems of Heged$\ddot{\mbox{u}}$s and Szil\'agyi \cite{n13} and Walter \cite{n14}, and subsequently introduced  the notion of weak quasicontraction mappings using the concept of comparison function and studied the existence of fixed point of weak quasicontraction mappings in the context of complete metric spaces. Hafshejani et al. \cite{n21} investigated the existence of fixed points and best proximity points of cyclic and noncyclic Fisher quasi-contraction mappings. Recently, Digar et al. \cite{n16} and Gabeleh \cite{n17} explored the best proximity point results for pointwise cyclic contraction mappings with respect to orbits and modified generalized pointwise cyclic contraction mappings with respect to orbits, respectively.   
 
Let $(\mathfrak{X},d)$ be a metric space. A path in $(\mathfrak{X},d)$ is a continuous mapping $\chi:[0,1] \rightarrow \mathfrak{X}$. A path $\chi: [0,1] \rightarrow \mathfrak{X}$ is a geodesic if $ d(\chi(\mu), \chi(\kappa))=\vert \mu-\kappa \vert d(\chi(0), \chi(1))$ for every $ \mu,\kappa \in [0,1]$. A metric space $(\mathfrak{X},d)$ is a   geodesic space whenever each pair of points $\mu,\kappa \in \mathfrak{X}$ can be joined by at least one geodesic such that $\chi(0)=\mu$ and $\chi(1)=\kappa$. The geodesic is denoted by $[\mu,\kappa]$. Geodesics determined  by its endpoints are not necessarily unique in general. For $\nu \in [\mu,\kappa]$, we write $\nu=(1-\lambda)\mu \oplus \lambda \kappa$, where $\lambda=\frac{d(\mu,\nu)}{d(\mu,\kappa)}$ and $\mu \neq \kappa$. If for each pair 
$\mu,\kappa \in \mathfrak{X}$, there is exactly one geodesic connecting them, then the metric space $(\mathfrak{X},d)$ is called uniquely geodesic. A subset $\mathfrak{M}$ of $\mathfrak{X}$ is called convex  if $[\mu,\kappa]\subset \mathfrak{M}$ for any $\mu,\kappa \in \mathfrak{M}$. Complete Riemannian manifolds and normed vector spaces are examples of geodesic spaces. 

A geodesic triangle $\bigtriangleup(\mu,\kappa,\nu)$ in a geodesic metric space $(\mathfrak{X},d)$ consists of three points (vertices) $\mu,\kappa,\nu$ in $\mathfrak{X}$ and a geodesic segment (edge) between each pair of vertices. For a geodesic triangle $\bigtriangleup(\mu,\kappa,\nu)$ in $(\mathfrak{X},d)$, a comparison triangle is $\overline{\bigtriangleup}(\mu,\kappa,\nu)=\bigtriangleup(\overline{\mu},\overline{\kappa},\overline{\nu})$ in the Banach space $l_p$, where $p \geq 2$ and $\Vert \overline{\mu}-\overline{\kappa}\Vert=d(\mu,\kappa)$, $\Vert \overline{\kappa}-\overline{\nu}\Vert=d(\kappa,\nu)$ and $\Vert \overline{\mu}-\overline{\nu}\Vert=d(\mu,\nu)$. A point $\overline{\xi} \in [\overline{\mu},\overline{\kappa}]$ is called a comparison point for $\xi \in [\mu,\kappa]$ if $d(\overline{\mu},\overline{\xi})=d(\mu,\xi)$. A metric space $(\mathfrak{X},d)$ is called a $CAT(0)$ space \cite{n18} if 

(i) $(\mathfrak{X},d)$ is geodesically connected,

(ii) every geodesic triangle in $\mathfrak{X}$ is atleast as thin as its comparison triangle in the Euclidean plane.
\begin{definition}
\emph{\cite[Definition 3.2]{n19} A geodesic metric space $(\mathfrak{X},d)$ is called a  $CAT_p(0)$ space, with $p \geq 2$ if for any geodesic triangle $\bigtriangleup$ in $\mathfrak{X}$, there exists a comparison triangle $\overline{\bigtriangleup}$ in $l_p$ such that for all $\mu,\kappa \in \bigtriangleup$ and $\overline{\mu},\overline{\kappa} \in \overline{\bigtriangleup}$, we have $d(\mu,\kappa)\leq \Vert \overline{\mu}-\overline{\kappa}\Vert$. }
\end{definition}
An example of a complete $CAT_p(0)$ space is the normed vector space $l_p$ ($p \geq 2$). In the setting of a complete $CAT_p(0)$ metric space, with $p \geq 2$, Shukri \cite{n20} obtained best proximity point results for cyclic contraction mappings with the aid of the following two lemmas:
\begin{lemma}\label{lemma 2.2}\cite[Lemma 3.1]{n20}
Let  $\Omega$ and $\Delta$ be nonempty and closed  subsets of a complete $CAT_{p}(0)$ metric space, with $p \geq 2$. Assume that $\Omega$ is convex.  Let $\{\varsigma_{r}\}$ and $ \{\varrho_{r}\}$ be sequences in $\Omega$ and $\{\vartheta_{r}\}$ be a sequence in $\Delta$ such that

  (i) For each $ \epsilon >0$, there exists $\mathcal{M}_{0}$ such that  for all $ s>r \geq \mathcal{M}_{0}$, $d(\varsigma_{s},\vartheta_{r})\leq dist(\Omega,\Delta)+\epsilon$,
  
  (ii) $ d(\varrho_{r},\vartheta_{r})\rightarrow dist(\Omega,\Delta)$.\\
 Then for each $ \epsilon >0$, there exists $\mathcal{M}_{1}$ such that for all $s > r \geq \mathcal{M}_{1}$,  $d(\varsigma_{s},\varrho_{r}) \leq \epsilon$ .
\end{lemma}
\begin{lemma}\label{lemma 2.3}\cite[Lemma 3.2]{n20}
Let  $\Omega$ and $\Delta$ be nonempty and closed  subsets of a complete $CAT_{p}(0)$ metric space, with $p \geq 2$. Assume that $\Omega$ is convex. Let $ \{\varsigma_{r}\}$ and $ \{\varrho_{r}\}$ be sequences in $\Omega$ and $\{\vartheta_{r}\}$ be a sequence in $\Delta$ such that
  
 (i) $d(\varsigma_{r},\vartheta_{r}) \rightarrow dist(\Omega,\Delta)$,
 
 (ii) $d(\varrho_{r},\vartheta_{r}) \rightarrow dist(\Omega,\Delta)$.\\
 Then $d(\varsigma_{r},\varrho_{r}) \rightarrow 0.$
 \end{lemma}
 \begin{theorem}
 \cite[Theorem 3.1]{n20} Let $\Omega$ and $\Delta$ be nonempty, closed and convex subsets of  a complete $CAT_{p}(0)$ metric space $(\mathfrak{X},d)$, with $p \geq 2$. Let $\Xi: \Omega \cup \Delta \rightarrow \Omega \cup \Delta$ be a cyclic contraction mapping. Then $\Xi$ has a unique best proximity point. 
 \end{theorem}
In this paper, we introduce the concept of cyclic orbital contraction mapping which is a generalization of the concept of cyclic contraction mapping. This type of contraction mapping enables us to consider distance between all the points in $\mathcal{O}_{\Xi}(\varsigma,\vartheta)$ such that one point lies in the set $\Omega$ and the other point lies in the set $\Delta$ and hence, aligns with the idea of cyclic mappings in best proximity point theory. We explore the existence of best proximity points of such mappings in the framework of $CAT_p(0)$ metric spaces. Moreover, in the case of a uniformly convex Banach space, we obtain the existence of a best proximity point.
\section{Main Results}
In this section, we define the concept of cyclic orbital contraction mappings and explore the best proximity point results of these mappings in the framework of $CAT_p(0)$ metric spaces and uniformly convex Banach spaces. We begin by stating the following definition: 
\begin{definition}
\emph{Let $\Omega$ and $\Delta$ be nonempty subsets of a metric space $(\mathfrak{X},d)$. A mapping $\Xi: \Omega \cup \Delta \rightarrow \Omega \cup \Delta$ is a cyclic orbital contraction mapping if for all $\varsigma \in \Omega$ and $\vartheta \in \Delta$,  we have}
    
    \emph{(i) $\Xi$ induces bounded orbits,}
    
    \emph{(ii) $ \Xi(\Omega) \subset \Delta$ and $\Xi(\Delta) \subset \Omega$,}
    
   \emph{(iii) $d(\Xi\varsigma,\Xi\vartheta) \leq \eta 
\sup\limits_{\substack{
i - j { \thinspace odd} \\
k - l {\thinspace odd} \\
p - q {\thinspace even}
}}
\{ d(\Xi^i\varsigma,\Xi^j\varsigma), d(\Xi^k\vartheta,\Xi^l\vartheta), d(\Xi^p\varsigma,\Xi^q\vartheta) \} + (1- \eta) dist(\Omega,\Delta)$ for some }

\emph{\hspace{0.75cm}$\eta \in (0,1)$.}
\end{definition}
The following example illustrates that there exists a mapping which is cyclic orbital contraction mapping but not cyclic contraction mapping:
\begin{example}\label{example2.2}
\emph{Let $\mathfrak{X}=\mathbb{R}^2$ be the Euclidean space. Let
\begin{align*}
\Omega&= \Big \{ (-1,-a): a \in \Big[-\frac{1}{2},\frac{1}{2}\Big] \Big \},\\
\Delta&= \Big \{ (1,-b): b \in \Big[-\frac{1}{2},\frac{1}{2}\Big] \Big \}.
\end{align*}  Then $dist(\Omega,\Delta)=2$. Define $\Xi:\Omega \cup \Delta \rightarrow \Omega \cup \Delta$ by
\begin{align*}
\Xi x=\Xi(x_1,x_2)=\left\{\begin{array}{lll}
\big(x_1,-\frac{x_2}{2} \big)+(2,0) & \textnormal{if} & x\in \Omega, \\
\big(x_1,-\frac{x_2}{3} \big)-(2,0) & \textnormal{if} & x \in \Delta. 
\end{array} \right.
\end{align*}
Observe that if $(-1,-a) \in \Omega$, then $\Xi(-1,-a)=(1,\frac{a}{2}) \in \Delta$ and if $(1,-b) \in \Delta$, then $\Xi(1,-b)=\\(-1,\frac{b}{3}) \in \Omega$. We see that there does not exists $\eta \in (0,1)$ such that  $$d\Big(\Xi\Big(-1,\frac{-1}{2}\Big),\Xi\Big(1,\frac{-1}{2}\Big)\Big)=\sqrt{4+\frac{1}{144}} \leq 2 \eta+(1-\eta)dist(\Omega,\Delta).$$ Therefore, $\Xi$ is not a cyclic contraction mapping. We see that $d(\Xi(-1,-a),\Xi(1,-b)) \leq 0.95 \max\{\\d((-1,-a),\Xi(-1,-a)),d((1,-b),\Xi(1,-b))\}+(1-0.95)dist(\Omega,\Delta)$  which gives 
\begin{align*}
d(\Xi(-1,-a),\Xi(1,-b)) & \leq 0.95 \sup\limits_{\substack{
i - j { \thinspace odd} \\
k - l {\thinspace odd} \\
p - q {\thinspace even}
}}
\{\{d(\Xi^i(-1,-a),\Xi^j(-1,-a)),d(\Xi^k(1,-b),\Xi^l(1,-b)),\\
& \thinspace \thinspace \quad d(\Xi^p(-1,-a),\Xi^q(1,-b))\}+(1-0.95)dist(\Omega,\Delta).
\end{align*}  
Therefore, $\Xi$ is a cyclic orbital  contraction mapping.}
   \end{example}
   The following result serves as a key tool in deriving the subsequent results:  
\begin{proposition}\label{proposition 2.4}
Let $\Omega$ and $\Delta$ be nonempty subsets of a $CAT_{p}(0)$ metric space $(\mathfrak{X},d)$, with $p \geq 2$. Let $ \Xi:\Omega \cup \Delta \rightarrow \Omega \cup \Delta$ be a cyclic orbital contraction mapping. Let $\varsigma_0 \in \Omega$. Define $\varsigma_{n+1}=\Xi\varsigma_{n}$ for each $ n \in \mathbb{N} \cup \{0\}$. Then $d(\varsigma_n,\varsigma_{n+1}) \rightarrow dist(\Omega,\Delta)$. 
\end{proposition}
\begin{proof}
    Consider
    \begin{align*}
       d(\varsigma_1,\varsigma_2) & = d(\Xi\varsigma_0,\Xi\varsigma_1) \\
        & \leq \eta  
        \sup_{\substack{
i - j {\thinspace odd} \\
k - l {\thinspace odd} \\
p - q {\thinspace even}
}} \{d(\Xi^i\varsigma_0,\Xi^j\varsigma_0), d(\Xi^k\varsigma_1,\Xi^l \varsigma_1), d(\Xi^p\varsigma_0,\Xi^q \varsigma_1)\} + (1- \eta) dist (\Omega,\Delta).
    \end{align*}
    Consider
\begin{align*}
   d(\varsigma_2,\varsigma_3) & = d(\Xi\varsigma_1,\Xi\varsigma_2)\\
    & \leq \eta 
    \sup_{\substack{
a - b {\thinspace odd} \\
c - d {\thinspace odd} \\
e - f {\thinspace even}
}} \{d(\Xi^a\varsigma_1,\Xi^b\varsigma_1), d(\Xi^c\varsigma_2,\Xi^d \varsigma_2), d(\Xi^e \varsigma_1,\Xi^f \varsigma_2) \} + (1- \eta) dist (\Omega,\Delta).
\end{align*}
We see that 
\begin{align*}
d(\Xi^a\varsigma_1,\Xi^b\varsigma_1) & =d(\Xi(\Xi^a\varsigma_0),\Xi(\Xi^b\varsigma_0))\\
& \leq \eta \sup_{\substack{
i - j {\thinspace odd} \\
k - l {\thinspace odd} \\
p - q { \thinspace even}
}} \{d(\Xi^i(\Xi^a\varsigma_0),\Xi^j(\Xi^a\varsigma_0)),d(\Xi^k(\Xi^b\varsigma_0),\Xi^l(\Xi^b\varsigma_0)),d(\Xi^p(\Xi^a\varsigma_0),\Xi^q(\Xi^b\varsigma_0))\}\\
& \quad +(1-\eta)dist(\Omega,\Delta).
\end{align*}
This implies that $d(\Xi^a\varsigma_1,\Xi^b\varsigma_1) \leq \eta \sup\limits_{\substack{
i - j {\thinspace odd}}} \{d(\Xi^i\varsigma_0,\Xi^j\varsigma_0)\}+(1-\eta)dist(\Omega,\Delta)$. Similarly, we show that $$d(\Xi^c\varsigma_2,\Xi^d\varsigma_2) \leq \eta \sup\limits_{\substack{
k - l {\thinspace odd}}} \{d(\Xi^k\varsigma_1,\Xi^l\varsigma_1)\}+(1-\eta)dist(\Omega,\Delta)$$ and $$d(\Xi^e\varsigma_1,\Xi^f\varsigma_2) \leq \eta \sup_{\substack{
i - j {\thinspace odd} \\
k - l {\thinspace odd} \\
p - q { \thinspace even}
}} \{d(\Xi^i\varsigma_0,\Xi^j\varsigma_0),d(\Xi^k\varsigma_1,\Xi^l\varsigma_1),d(\Xi^p\varsigma_0,\Xi^q\varsigma_1)\}+(1-\eta)dist(\Omega,\Delta).$$
This gives 
\begin{align*}
    d(\varsigma_2,\varsigma_3) & \leq \eta\Big[\eta 
    \sup_{\substack{
i - j {\thinspace odd} \\
k - l {\thinspace odd} \\
p - q { \thinspace even}
}} \{d(\Xi^i \varsigma_0,T^j \varsigma_0), d(\Xi^k \varsigma_1,\Xi^l \varsigma_1), d(\Xi^p \varsigma_0,\Xi^q \varsigma_1)\}\\
& \quad +(1-\eta)dist (\Omega,\Delta)\Big] + (1- \eta) dist (\Omega,\Delta)
\end{align*}
which implies that
\begin{align*}
    d(\varsigma_2,\varsigma_3) & \leq \eta^2  
    \sup_{\substack{
i - j {\thinspace odd} \\
k - l {\thinspace odd} \\
p - q { \thinspace even}
}} \{d(\Xi^i \varsigma_0,T^j \varsigma_0), d(\Xi^k \varsigma_1,\Xi^l \varsigma_1), d(\Xi^p \varsigma_0,\Xi^q \varsigma_1)\} + (1- \eta^2) dist (\Omega,\Delta).
\end{align*}
Similarly, we can prove 
$$ d(\varsigma_3,\varsigma_4) \leq \eta^3 
\sup_{\substack{
i - j {\thinspace odd} \\
k - l {\thinspace odd} \\
p - q {\thinspace even}}} \{d(\Xi^i \varsigma_0,\Xi^j \varsigma_0), d(\Xi^k \varsigma_1,\Xi^l \varsigma_1), d(\Xi^p \varsigma_0,\Xi^q \varsigma_1)\} + (1- \eta^3) dist (\Omega,\Delta).$$
Proceeding on the same lines, we obtain
$$d(\varsigma_n, \varsigma_{n+1}) \leq \eta^n 
\sup_{\substack{
i - j {\thinspace odd} \\
k - l {\thinspace odd} \\
p - q {\thinspace even}
}}\{d(\Xi^i \varsigma_0,\Xi^j \varsigma_0), d(\Xi^k \varsigma_1,\Xi^l \varsigma_1), d(\Xi^p \varsigma_0,\Xi^q \varsigma_1) \} + (1- \eta^n) dist (\Omega,\Delta).$$
Letting $n \rightarrow \infty$, we get $d(\varsigma_n, \varsigma_{n+1}) \rightarrow dist(\Omega,\Delta)$.
\end{proof}
We now proceed with the formulation of the first best proximity point result. 
  \begin{theorem}\label{theorem 2.6}
    Let $\Omega$ and $\Delta$ be nonempty, closed and convex subsets of a complete $CAT_{p}(0)$ metric space $(\mathfrak{X},d)$, with $p \geq 2$. Let $\Xi:\Omega \cup \Delta \rightarrow \Omega \cup \Delta$ be a cyclic orbital contraction mapping. Then
    
    (i) $\Xi$ has a best proximity point $\varsigma^* \in \Omega$,
    
    (ii) $\varsigma^*$ is the unique fixed point of $\Xi^2$,
    
    (iii) $\{\Xi^{2n}\varsigma\}$ converges to $\varsigma^*$ for every $\varsigma \in \Omega$,
    
    (iv) $\Xi\varsigma^*$ is a best proximity point of $\Xi$ in $\Delta$ and $\{\Xi^{2n}\vartheta \}  $ converges to $\Xi\varsigma^*$ for every $\vartheta \in B$.
\end{theorem}
 \begin{proof}
     Let $\varsigma_0 \in \Omega$. Define $\varsigma_{n+1}= \Xi\varsigma_{n}$ for each $n \in \mathbb{N} \cup \{0\}$. Using  Proposition \ref{proposition 2.4}, we get $d(\varsigma_{2n}, \varsigma_{2n+1}) \rightarrow dist (\Omega,\Delta)$. Similarly, $d(\varsigma_{2n+1},\varsigma_{2n+2}) \rightarrow dist (\Omega,\Delta)$. Using Lemma \ref{lemma 2.3}, we infer that $d(\varsigma_{2n},\varsigma_{2n+2}) \rightarrow 0$. Also, by Proposition \ref{proposition 2.4}, $d(\varsigma_{2n+1},\varsigma_{2n+2}) \rightarrow dist (\Omega,\Delta)$ and $d(\varsigma_{2n+2},\varsigma_{2n+3}) \rightarrow dist(\Omega,\Delta)$. Using Lemma \ref{lemma 2.3}, we infer that $d(\varsigma_{2n+1},\varsigma_{2n+3}) \rightarrow 0$. 
    Now, we want to prove that $\{\varsigma_{2n}\}$ is a Cauchy sequence in $\Omega$. For this, we prove that for every $ \epsilon>0$, there exists $\tilde{\mathcal{N}}$ such that for all $ m>n \geq \tilde{\mathcal{N}}$, $d(\varsigma_{2m},\varsigma_{2n+1}) \leq dist (\Omega,\Delta)+ \epsilon$. 
    
    For $m>n$, consider
     \begin{align}\label{equation 1.1}
    d(\varsigma_{2m},\varsigma_{2n+1}) & = d(\varsigma_{2m-2n+2n},  \varsigma_{2n+1}) \nonumber \\
     &=d(\Xi^{2n}\varsigma_{2m-2n},\Xi^{2n}\varsigma_1)\nonumber \\
      &\leq \eta^{2n} 
      \sup_{\substack{
i - j {\thinspace odd} \\
k - l {\thinspace odd} \\
p - q {\thinspace even}
}}\{d(\Xi^{i}\varsigma_{2m-2n},\Xi^{j}\varsigma_{2m-2n}),d(\Xi^{k}\varsigma_1,\Xi^{l}\varsigma_1), d(\Xi^{p}\varsigma_{2m-2n},\Xi^{q}\varsigma_1) \} \nonumber \\
&\quad + (1-\eta^{2n}) dist (A,B).
 \end{align}
 As $\eta^{2n} \rightarrow 0$ as $n \rightarrow \infty$ and $\Xi$ induces bounded orbits, there exists $\mathcal{N}_0 \in \mathbb{N}$ such that for all $n \geq \mathcal{N}_0$, we have
 $$ \eta^{2n}
 \sup_{\substack{
i - j {\thinspace odd} \\
k - l {\thinspace odd} \\
p - q {\thinspace even}
}}\{d(\Xi^{i}\varsigma_{2m-2n},\Xi^{j}\varsigma_{2m-2n}),d(\Xi^{k}\varsigma_1,\Xi^{l}\varsigma_1), d(\Xi^{p}\varsigma_{2m-2n},\Xi^{q}\varsigma_1)\} < \frac{\epsilon}{2}. $$ 
 As $\eta^{2n} \rightarrow 0$ as $n \rightarrow \infty$, 
 $(1-\eta^{2n}) dist (\Omega,\Delta) \rightarrow dist(\Omega,\Delta)$.  Therefore, there exists $\mathcal{N}_1 \in \mathbb{N}$ such that $(1-\eta^{2n}) dist (\Omega,\Delta) < \frac{\epsilon}{2}+ dist(\Omega,\Delta)$ for all $n \geq \mathcal{N}_1.$ Let $\tilde{\mathcal{N}}=\max \{\mathcal{N}_0,\mathcal{N}_1 \}$. Therefore, using (\ref{equation 1.1}), for all $ m>n \geq \tilde{\mathcal{N}}$, we get
$$d(\varsigma_{2m},\varsigma_{2n+1}) < \frac{\epsilon}{2}+ \frac{\epsilon}{2} +dist(\Omega,\Delta).$$
This implies that $d(\varsigma_{2m},\varsigma_{2n+1})<\epsilon+ dist(\Omega,\Delta)$ for all $ m>n \geq \tilde{\mathcal{N}}$. Also, $d(\varsigma_{2n}, \varsigma_{2n+1}) \rightarrow dist (\Omega,\Delta)$. Therefore, using  Lemma \ref{lemma 2.2} for every $ \epsilon>0$, there exists $\hat{\mathcal{N}}$ such that for all $ m>n \geq \hat{\mathcal{N}}$, we have $d(\varsigma_{2m},\varsigma_{2n})\leq \epsilon$. This implies that $\{\varsigma_{2n}\}$ is a Cauchy sequence in $\Omega$. Since $\Omega$ is complete, there exists $\varsigma^* \in \Omega$ such that $\varsigma_{2n} \rightarrow \varsigma^*$.

(i) We claim that $d(\Xi^{2n+1}\varsigma^*,\varsigma^*) \rightarrow dist(A,B)$. As
\begin{align*}
    dist(\Omega,\Delta) & \leq d(\varsigma^*,\Xi^{2n+1}\varsigma^*)\\
    & \leq d(\varsigma^*,\varsigma_{2n}) + d(\varsigma_{2n},\Xi^{2n+1}\varsigma^*)\\
    &= d(\varsigma^*,\varsigma_{2n}) +d(\Xi^{2n}\varsigma_0,\Xi^{2n+1}\varsigma^*)\\
    & \leq d(\varsigma^*,\varsigma_{2n}) + \eta^{2n} 
    \sup_{\substack{
i - j {\thinspace odd} \\
k - l {\thinspace odd} \\
p - q {\thinspace even}
}}\{ d(\Xi^i\varsigma_0,\Xi^j\varsigma_0), d(\Xi^k(\Xi\varsigma^*),\Xi^l(\Xi\varsigma^*)), d(\Xi^p\varsigma_0,\Xi^q(\Xi\varsigma^*)) \} \\
&\quad + (1-\eta^{2n}) dist(\Omega,\Delta).
\end{align*}
Letting $ n \rightarrow \infty$, we get $d(\varsigma^*,\Xi^{2n+1}\varsigma^*) \rightarrow dist (\Omega,\Delta)$. We want to show that $ 
\sup\limits_{\substack{
i - j {\thinspace odd} \\
}}\{d(\Xi^i\varsigma^*,\Xi^j\varsigma^*)\}= dist (\Omega,\Delta)$. On the contrary, assume that $ \sup\limits_{\substack{
i - j {\thinspace odd} \\
}} \{d(\Xi^i\varsigma^*,\Xi^j\varsigma^*)\}> dist (\Omega,\Delta)$. If $r$ and $s$ are odd, consider
\begin{align*}
d(\Xi^r\varsigma^*,\Xi^{r+s}\varsigma^*) & = d( \Xi^r\varsigma^*,\Xi^r(\Xi^s\varsigma^*))\\
    & \leq \eta^r 
  \sup_{\substack{
i - j {\thinspace odd} \\
k - l {\thinspace odd} \\
p - q {\thinspace even}
}}  \{d(\Xi^i\varsigma^*,\Xi^j\varsigma^*),d(\Xi^k(\Xi^s\varsigma^*),\Xi^l(\Xi^s\varsigma^*)),d(\Xi^p\varsigma^*,\Xi^q(\Xi^s\varsigma^*))\}\\
& \quad +(1- \eta^r) dist(\Omega,\Delta)\\
& = \eta^r 
 \sup_{\substack{
i - j {\thinspace odd} \\
}}\{d(\Xi^i\varsigma^*,\Xi^j\varsigma^*)\} +(1- \eta^r) dist(\Omega,\Delta)\\
 & < \eta^r 
 \sup_{\substack{
i - j {\thinspace odd} \\
}}\{d(\Xi^i\varsigma^*,\Xi^j\varsigma^*)\} +(1- \eta^r) 
\sup_{\substack{
i - j {\thinspace odd} \\
}}\{d(\Xi^i\varsigma^*,\Xi^j\varsigma^*)\}.
\end{align*}
This implies that $d(\Xi^r\varsigma^*,\Xi^{r+s}\varsigma^*) < \sup\limits_{\substack{i - j {\thinspace odd} \\
}}\{d(\Xi^i\varsigma^*,\Xi^j\varsigma^*)\}$. If $r$ is even and $s$ is odd, consider
\begin{align*}     
d(\Xi^r\varsigma^*,\Xi^{r+s}\varsigma^*) & \leq \eta^r 
    \sup_{\substack{
i - j {\thinspace odd} \\
k - l {\thinspace odd} \\
p - q {\thinspace even}
}} \{d(\Xi^i\varsigma^*,\Xi^j\varsigma^*), d(\Xi^k(\Xi^s\varsigma^*),\Xi^l(\Xi^s\varsigma^*)),d(\Xi^p\varsigma^*,\Xi^q(\Xi^s\varsigma^*))\} \\
& \quad +(1- \eta^r) dist(\Omega,\Delta)\\
   & =  \eta^r    \sup_{\substack{
i - j {\thinspace odd} \\
}}\{d(\Xi^i\varsigma^*,\Xi^j\varsigma^*)\} +(1- \eta^r) dist(\Omega,\Delta)\\
 & < \eta^r 
  \sup_{\substack{
i - j {\thinspace odd} \\
}}\{d(\Xi^i\varsigma^*,\Xi^j\varsigma^*)\} +(1- \eta^r)
 \sup_{\substack{
i - j {\thinspace odd} \\
}}\{d(\Xi^i\varsigma^*,\Xi^j\varsigma^*)\}.
\end{align*}
 This implies that $d(\Xi^r\varsigma^*,\Xi^{r+s}\varsigma^*) < \sup\limits_{\substack{i - j {\thinspace odd} \\
}}\{d(\Xi^i\varsigma^*,\Xi^j\varsigma^*)\}$. Therefore, $ \sup\limits_{\substack{
i - j {\thinspace odd} \\
}}\{d(\Xi^i\varsigma^*,\Xi^j\varsigma^*)\}= 
 \sup\limits_{\substack{
n {\thinspace odd} \\
}}\{d(\varsigma^*,\\ \Xi^n\varsigma^*)\}$. As $d(\varsigma^*,\Xi^n\varsigma^*)\rightarrow dist(\Omega,\Delta)$, where $n$ is odd, there exists $n_0 \in \mathbb{N}$ such that 
\begin{equation*}
\begin{split}
 \sup_{\substack{
i - j {\thinspace odd} \\
}} \{d(\Xi^i\varsigma^*,\Xi^j\varsigma^*)\}&=\left\{\begin{array}{ll}
\max \{d(\varsigma^*,\Xi^r\varsigma^*):& r=1,3,\ldots, n_0  \mbox{ if } n_0 \mbox{ is odd}\},\\
\max \{d(\varsigma^*,\Xi^r\varsigma^*):& r=1,3,\ldots,n_0-1  \mbox{ if } n_0 \mbox{ is even}\}.\\
\end{array}
\right.\\\ 
\end{split}
\end{equation*}
Let $a$ be the index where the maximum is attained. 
\begin{align*}
   d(\varsigma^*,\Xi^a\varsigma^*) & \leq d(\varsigma^*,\Xi^{2n}\varsigma^*)+ d(\Xi^{2n}\varsigma^*,\Xi^a\varsigma^*)\\
    & = d(\varsigma^*,\Xi^{2n}\varsigma^*)+d(\Xi^a(\Xi^{2n-a}\varsigma^*),\Xi^a\varsigma^*)\\
    & \leq d(\varsigma^*,\Xi^{2n}\varsigma^*)+ \eta^a
     \sup_{\substack{
i - j {\thinspace odd} \\
k - l {\thinspace odd} \\
p - q {\thinspace even}
}} \{d(\Xi^i(\Xi^{2n-a}\varsigma^*),\Xi^j(\Xi^{2n-a}\varsigma^*)), d(\Xi^k\varsigma^*,\Xi^l\varsigma^* ),d(\Xi^p(\Xi^{2n-a}\varsigma^*),\Xi^q\varsigma^*)\}\\
 & \quad + (1- \eta^a) dist(\Omega,\Delta).
\end{align*}
Therefore,
\begin{equation}\label{equation 1.2}
d(\varsigma^*,\Xi^a\varsigma^*)\leq d(\varsigma^*,\Xi^{2n}\varsigma^*)+ \eta^a 
  \sup_{\substack{
i - j {\thinspace odd} \\
}} \{ d(\Xi^i\varsigma^*,\Xi^j\varsigma^*)\} +(1- \eta^a) dist (\Omega,\Delta). 
\end{equation}
Consider
\begin{align*}
 d(\Xi^{2n}\varsigma^*,\Xi^{2n+1}\varsigma^*) & \leq \eta^{2n} 
     \sup_{\substack{
i - j {\thinspace odd} \\
k - l {\thinspace odd} \\
p - q {\thinspace even}
}}\{d(\Xi^i\varsigma^*,\Xi^j\varsigma^*),d( \Xi^k(\Xi\varsigma^*),\Xi^l(\Xi\varsigma^*)),d(\Xi^p\varsigma^*,\Xi^q(\Xi\varsigma^*) )\}\\
& \quad + (1- \eta^{2n}) dist(\Omega,\Delta).
\end{align*}
This gives $ d(\Xi^{2n}\varsigma^*,\Xi^{2n+1}\varsigma^*)  \leq \eta^{2n} 
     \sup\limits_{\substack{
i - j {\thinspace odd} 
}}\{d(\Xi^i\varsigma^*,\Xi^j\varsigma^*)\}+ (1- \eta^{2n}) dist(\Omega,\Delta)$. Letting $ n \rightarrow \infty$, we get $d(\Xi^{2n}\varsigma^*,\Xi^{2n+1}\varsigma^*)\rightarrow dist(\Omega,\Delta)$. Also, $d(\varsigma^*,\Xi^{2n+1}\varsigma^*)\rightarrow dist(\Omega,\Delta)$. Therefore, by Lemma \ref{lemma 2.3}, we get $d(\varsigma^*,\Xi^{2n}\varsigma^*)\rightarrow 0$. Letting $n \rightarrow \infty$ in ($ \ref{equation 1.2}$), we get
   \begin{align*}
d(\varsigma^*,\Xi^a\varsigma^*)& \leq \eta^a 
     \sup_{\substack{
i - j {\thinspace odd} \\
}} \{d(\Xi^i\varsigma^*,\Xi^j\varsigma^*)\}+(1-\eta^a)dist(\Omega,\Delta)\\
    & < \eta^a 
     \sup_{\substack{
i - j {\thinspace odd} \\
}} \{d(\Xi^i\varsigma^*,\Xi^j\varsigma^*)\}+(1-\eta^a)
 \sup_{\substack{
i - j {\thinspace odd} \\
}} \{d(\Xi^i\varsigma^*,\Xi^j\varsigma^*)\}.
\end{align*}
Therefore, $d(\varsigma^*,\Xi^a\varsigma^*) <  \sup\limits_{\substack{
i - j {\thinspace odd}\\ 
}}  \{d(\Xi^i\varsigma^*,\Xi^j\varsigma^*)\}$ which gives
$\sup\limits_{\substack{
i - j {\thinspace odd} \\
}}  \{d(\Xi^i\varsigma^*,\Xi^j\varsigma^*)\} <   \sup\limits_{\substack{
i - j {\thinspace odd} \\
}}  \{d(\Xi^i\varsigma^*,\Xi^j\varsigma^*)\}$, a contradiction. Therefore, $\sup\limits_{\substack{
i - j {\thinspace odd} \\
}}  \{d(\Xi^i\varsigma^*,\Xi^j\varsigma^*)\}=dist(\Omega,\Delta)$. This implies $d(\varsigma^*,\Xi\varsigma^*)= dist(\Omega,\Delta)$.  Therefore $\varsigma^*$ is a best proximity point of $\Xi$.

 (ii) We need to show that $\varsigma^*$ is the unique fixed point of $\Xi^2$.   Consider 
 \begin{align*}
d(\Xi\varsigma^*,\Xi^2\varsigma^*)& \leq\eta 
     \sup_{\substack{
i - j {\thinspace odd} \\
k - l {\thinspace odd} \\
p - q {\thinspace even}
}}\{d(\Xi^i\varsigma^*,\Xi^j\varsigma^*),d(\Xi^k(\Xi\varsigma^*),\Xi^l(\Xi\varsigma^*)),d(\Xi^p\varsigma^*,\Xi^q(\Xi\varsigma^*))) \} + (1- \eta) dist(\Omega,{B}).
\end{align*}
Therefore,
\begin{align*}
   d(\Xi\varsigma^*,\Xi^2\varsigma^*)  & \leq \eta  \sup_{\substack{
i - j {\thinspace odd} \\
}} \{d(\Xi^i\varsigma^*,\Xi^j\varsigma^*)\}+(1- \eta) dist(\Omega,\Delta)\\
     & =\eta dist(\Omega,\Delta)+(1- \eta) dist(\Omega,\Delta)
 \end{align*}
which implies that $d(\Xi\varsigma^*,\Xi^2\varsigma^* )\leq dist(\Omega,\Delta)$. This gives $d(\Xi\varsigma^*,\Xi^2\varsigma^*)= dist(\Omega,\Delta)$. Also, $d(\varsigma^*,\Xi\varsigma^*)= dist(\Omega,\Delta)$. Using Lemma \ref{lemma 2.3}, we get $d(\varsigma^*,\Xi^2\varsigma^*)=0$ which gives $\varsigma^*$ is a fixed point of $\Xi^2$. Now, it remains to prove the uniqueness of fixed point of $\Xi^2$. Let $\varrho$ be the another fixed point of $\Xi^2$.  First, we show that $\varrho$ is a best proximity point of $\Xi$. As $\Xi^2 \varrho=\varrho$. This implies that $\Xi^4\varrho=\Xi^2(\Xi^2\varrho)=\Xi^2\varrho=\varrho$. Therefore, $\Xi^{2n}\varrho=\varrho$ which gives $d(\varrho,\Xi^{2n}\varrho)=0$. We claim $d(\Xi^{2n+1}\varrho,\varrho) \rightarrow dist(\Omega,\Delta)$. 
 As \begin{align*}
           dist(A,B) & \leq d(\varrho, \Xi^{2n+1}\varrho)\\
     & \leq d(\varrho,\Xi^{2n}\varrho)+d(\Xi^{2n}\varrho,\Xi^{2n+1}\varrho)\\
     &\leq d(\varrho, \Xi^{2n}\varrho) + \eta^{2n}
      \sup_{\substack{
i - j {\thinspace odd} \\
k - l {\thinspace odd} \\
p - q {\thinspace even}
}}\{d(\Xi^i\varrho,\Xi^j\varrho) ,d(\Xi^k(\Xi\varrho),\Xi^l(\Xi\varrho) ),d(\Xi^p\varrho,\Xi^q(\Xi\varrho)) \}\\& \quad +(1- \eta^{2n}) dist(\Omega,\Delta).
\end{align*}
Therefore, $dist(A,B)  \leq d(\varrho, \Xi^{2n+1}\varrho) \leq d(\varrho,\Xi^{2n}\varrho) + \eta^{2n}   \sup\limits_{\substack{
i - j {\thinspace odd} \\
}}  \{d(\Xi^i\varrho,\Xi^j\varrho)\} + (1-\eta^{2n})dist(\Omega,\Delta)$.  Letting $ n \rightarrow \infty$, we get $d(\varrho, \Xi^{2n+1}\varrho) \rightarrow dist(\Omega,\Delta)$. Proceeding as in (i), we show that $\sup\limits_{\substack{
i - j {\thinspace odd} \\
}} d(\Xi^i\varrho,\Xi^j\varrho)=dist(\Omega,\Delta)$. Therefore, $\varrho$ is a best proximity point of $\Xi$.  If we assume that $d(\Xi\varsigma^*,\varrho) \leq  d(\varsigma^*,\Xi\varrho)$, consider \begin{align*}
d(\varsigma^*,\Xi\varrho) &= d(\Xi^2\varsigma^*,\Xi\varrho)\\
  & \leq \eta
    \sup_{\substack{
i - j {\thinspace odd} \\
k - l {\thinspace odd} \\
p - q {\thinspace even}
}}\{d(\Xi^i(\Xi\varsigma^*),\Xi^j(\Xi\varsigma^*)), d(\Xi^k\varrho,\Xi^l\varrho), d(\Xi^p(\Xi\varsigma^*),\Xi^q\varrho)\}  + (1- \eta) dist (\Omega,\Delta).
 \end{align*}
We see that $d(\Xi^i(\Xi\varsigma^*),\Xi^j(\Xi\varsigma^*)) = d(\varsigma^*,\Xi\varsigma^*) =dist(\Omega,\Delta)$ if $i-j$ is odd and $d(\Xi^k\varrho,\Xi^l\varrho) =dist (\Omega,\Delta)$ if $k-l$ is odd. Also, $d(\Xi^p(\Xi\varsigma^*),\Xi^q\varrho)= d(\varsigma^*,\Xi\varrho)$ if $p$ and $q$ are odd and $d(\Xi^p(\Xi\varsigma^*),\Xi^q\varrho) = d(\Xi\varsigma^*,\varrho)$ if $p$ and $q$ are even.
Therefore,
 \begin{align*}
d(\varsigma^*,\Xi\varrho) & \leq \eta \max \{dist(\Omega,\Delta), d(\varsigma^*,\Xi\varrho), d(\Xi\varsigma^*,\varrho) \} +(1- \eta) dist(\Omega,\Delta)\\
     &\leq \eta d(\varsigma^*,\Xi\varrho) +(1- \eta) dist(\Omega,\Delta).
\end{align*}
This implies $d(\varsigma^*,\Xi\varrho) =dist(\Omega,\Delta)$. As $d(\varrho,\Xi\varrho)= dist(\Omega,\Delta)$. Using Lemma \ref{lemma 2.3}, we infer that $d(\varsigma^*,\varrho)=0$ which gives $\varsigma^*=\varrho$. If we assume $d(\varsigma^*,\Xi\varrho)\leq d(\Xi\varsigma^*,\varrho)$, then proceeding as above, we get $\varsigma^*=\varrho$. Hence, $\varsigma^*$ is the unique fixed point of $\Xi^2$.

(iii) Let $\varsigma \in \Omega$. Then $\{\Xi^{2n}\varsigma\}$ is a Cauchy sequence in $\Omega$. Since $\Omega$ is closed and $\mathfrak{X}$ is complete, $\Omega$ is complete. Therefore, there exists $\varpi \in \Omega$ such that $\Xi^{2n} \varsigma \rightarrow \varpi$, where $\varpi$ is a best proximity point of $\Xi$. To  show $ \varsigma^*=\varpi$. We see that $\Xi^{2n+1}\varsigma_0 \rightarrow \Xi\varsigma^*$. Consider
 \begin{align*}
d(\Xi^{2n}\varsigma_0,\Xi^{2n+1}\varsigma_0)& \leq  \eta^{2n}
      \sup_{\substack{
i - j {\thinspace odd} \\
k - l {\thinspace odd} \\
p - q {\thinspace even}
}}\{d(\Xi^i\varsigma_0,\Xi^j\varsigma_0), d( \Xi^k(\Xi\varsigma_0),\Xi^l(\Xi\varsigma_0)), d(\Xi^p\varsigma_0,\Xi^q(\Xi\varsigma_0)) \} \\
& \quad+(1-\eta^{2n})dist(\Omega,\Delta).
\end{align*}
Therefore, $d(\Xi^{2n}\varsigma_0,\Xi^{2n+1}\varsigma_0) \leq \eta^{2n}
 \sup\limits_{\substack{
i - j {\thinspace odd} \\
}} 
\{ d(\Xi^i\varsigma_0,\Xi^j\varsigma_0)\} + (1-\eta^{2n})dist(\Omega,\Delta)$. Letting $n \rightarrow \infty$, we get $d(\Xi^{2n}\varsigma_0,\Xi^{2n+1}\varsigma_0) \rightarrow dist(\Omega,\Delta)$.  As $d(\varsigma^*,\Xi\varsigma^*)= dist(\Omega,\Delta)$. This implies $d( \Xi^{2n}\varsigma_0,\Xi\varsigma^*) \rightarrow dist(\Omega,\Delta)$. Using Lemma \ref{lemma 2.3}, we infer that 
$d(\Xi^{2n+1}\varsigma_0,\Xi\varsigma^*)\rightarrow 0$. Consider 
\begin{align*}
d(\varpi, \Xi\varsigma^*)&= \lim_{n \rightarrow \infty} d(\Xi^{2n}\varsigma,\Xi^{2n+1}\varsigma_0) \\
    & \leq \lim_{n \rightarrow \infty} \Big[ \eta^{2n} 
    \sup_{\substack{
i - j {\thinspace odd} \\
k - l {\thinspace odd} \\
p - q {\thinspace even}
}}\{d(\Xi^i\varsigma,\Xi^j\varsigma), d(\Xi^k(\Xi\varsigma_0),\Xi^l(\Xi\varsigma_0)), d(\Xi^p\varsigma,\Xi^q(\Xi\varsigma_0))\} +(1-\eta^{2n})dist(\Omega,\Delta) \Big ].
\end{align*}
Therefore,
$$d(\varpi, \Xi\varsigma^*) \leq \lim_{n \rightarrow \infty} \Big[ \eta^{2n} \sup_{\substack{
i - j {\thinspace odd} \\
k - l {\thinspace odd} \\
p - q {\thinspace even}
}} \{d(\Xi^i\varsigma,\Xi^j\varsigma), d(\Xi^k\varsigma_0,\Xi^l\varsigma_0), d(\Xi^p\varsigma,\Xi^q(\Xi\varsigma_0)) \} +(1-\eta^{2n})dist(\Omega,\Delta) \Big ].$$
Letting $ n \rightarrow \infty $, we get $d(\varpi,\Xi\varsigma^*)=dist(\Omega,\Delta)$. Also, $d(\varsigma^*, \Xi\varsigma^*)=dist(\Omega,\Delta)$. Using Lemma \ref{lemma 2.3}, we deduce that $d(\varsigma^*,\varpi)=0$.  Hence, $\{\Xi^{2n}\varsigma \}$ converges to $\varsigma^*$ for every $\varsigma \in \Omega$.
  
  (iv) Firstly, we show that $\{ \Xi^{2n}\vartheta \}$ converges to $\Xi\varsigma^*$ for every $\vartheta \in \Delta$. We see that $\Xi\vartheta \in \Omega$. Using (iii),  we get $\Xi^{2n}(\Xi\vartheta) \rightarrow \varsigma^*$ for every $\vartheta \in \Delta$. Consider \begin{align*}
d(\Xi^{2n}\vartheta,\varsigma^*)&= \lim_{n \rightarrow \infty}  d(\Xi^{2n}\vartheta, \Xi^{2n}(\Xi\vartheta))\\
    & \leq  \lim_{n \rightarrow \infty} \Big[ \eta^{2n} \sup_{\substack{
i - j {\thinspace odd} \\
k - l {\thinspace odd} \\
p - q {\thinspace even}
}} \{d(\Xi^i\vartheta,\Xi^j\vartheta), d(\Xi^k(\Xi\vartheta),\Xi^l(\Xi\vartheta)), d(\Xi^p\vartheta,\Xi^q(\Xi\vartheta))\} +(1-\eta^{2n})dist(\Omega,\Delta) \Big].
\end{align*}
Therefore,
$d(\Xi^{2n}\vartheta,\varsigma^*)\leq \lim\limits_{n \rightarrow \infty} [ \eta^{2n} 
\sup\limits_{\substack{
i - j {\thinspace odd} \\
}}\{d(\Xi^i\vartheta,\Xi^j\vartheta)\} +(1-\eta^{2n})dist(\Omega,\Delta)]$. This implies $d(\Xi^{2n}\vartheta,\varsigma^*)\rightarrow dist(\Omega,\Delta)$. Also, $d(\varsigma^*,\Xi\varsigma^*)= dist(\Omega,\Delta)$. Using Lemma \ref{lemma 2.3}, we infer that $d(\Xi^{2n}\vartheta,\Xi\varsigma^*)\rightarrow 0$. We show that $\Xi\varsigma^*$ is a best proximity point of $\Xi$ in $\Delta$. We claim that $d(\Xi^{2n+1}(\Xi\varsigma^*),\Xi\varsigma^*) \rightarrow dist(\Omega,\Delta)$. Consider
 \begin{align*}
          dist(\Omega,\Delta) & \leq d(\Xi^{2n+1}(\Xi\varsigma^*),\Xi\varsigma^*)\\
     & \leq d(\Xi^{2n+2}\varsigma^*,\Xi^{2n}\vartheta)+d(\Xi^{2n}\vartheta,\Xi\varsigma^*)\\
     & \leq \eta^{2n} \sup_{\substack{
i - j {\thinspace odd} \\
k - l {\thinspace odd} \\
p - q {\thinspace even}
}} \{d(\Xi^i(\Xi^2\varsigma^*),\Xi^j(\Xi^2\varsigma^*)), d(\Xi^k\vartheta,\Xi^l\vartheta), d(\Xi^p(\Xi^2\varsigma^*),\Xi^q\vartheta)\}\\
& \quad  + (1-\eta^{2n})dist(\Omega,\Delta) + d(\Xi^{2n}\vartheta,\Xi\varsigma^*).
\end{align*}
Letting $n \rightarrow \infty$, we get $d(\Xi^{2n+1}(\Xi\varsigma^*),\Xi\varsigma^*) \rightarrow dist(\Omega,\Delta)$. By (i), we have $ \sup\limits_{\substack{
i - j {\thinspace odd} \\
}} \{d(\Xi^i\varsigma^*,\Xi^j\varsigma^*)\}=dist(\Omega,\Delta)$. Therefore, $d(\Xi\varsigma^*,\Xi^2\varsigma^*)=dist(\Omega,\Delta)$ which gives $\Xi\varsigma^*$ is a best proximity point of $\Xi$ in $\Delta$. Now, it remains to show that $\{\Xi^{2n}\omega\}$ converges to $\Xi\varsigma^*$ for every $\omega \in \Delta$. Let $ \omega \in \Delta$. Then $\{\Xi^{2n}\omega\}$ is a Cauchy sequence in $\Delta$. Since $\Delta$ is complete, $\Xi^{2n} \omega \rightarrow \Omega$, where $\Omega$ is a best proximity point of $\Xi$ in $\Delta$. This implies that $d(\Omega,\Xi\Omega)= dist(\Omega,\Delta)$. We show that $\Xi\varsigma^*=\Omega$. We see that $\Xi^{2n+1} \omega \rightarrow \Xi \Omega$. 
Consider
 \begin{align*}
d(\Omega,\Xi^{2n+1}\omega)& = \lim_{n \rightarrow\infty} d(\Xi^{2n}\omega,\Xi^{2n+1}\omega)\\
    & \leq \lim_{n \rightarrow\infty} \Big [ \eta^{2n} \sup_{\substack{
i - j {\thinspace odd} \\
k - l {\thinspace odd} \\
p - q {\thinspace even}
}} \{d(\Xi^i \omega,\Xi^j \omega), d(\Xi^k(\Xi \omega),\Xi^l(\Xi \omega)), d(\Xi^p \omega,\Xi^q (\Xi\omega)) \} \\
& \quad +(1-\eta^{2n})dist(\Omega,\Delta) \Big ].
\end{align*}
 Therefore, $d(\Omega,\Xi^{2n+1}\omega) \leq  \lim\limits_{n \rightarrow\infty} [ \eta^{2n} \sup\limits_{\substack{
i-j{\thinspace odd}
}} \{d(\Xi^i \omega,\Xi^j \omega) \}+ (1-\eta^{2n})dist(\Omega,\Delta)]$. Thus, $d(\Omega,\Xi^{2n+1} \omega)\rightarrow dist(\Omega,\Delta)$. As $ d(\Omega,\Xi\Omega)=dist(\Omega,\Delta)$. Using  Lemma \ref{lemma 2.3}, we deduce that $d(\Xi^{2n+1} \omega,\Xi\Omega)\rightarrow 0$.  Consider
 \begin{align*}
d(\Xi\varsigma^*,\Xi\Omega)&= \lim_{n \rightarrow\infty} d(\Xi^{2n}\vartheta,T^{2n+1}\omega)\\
     & \leq  \lim_{n \rightarrow\infty} \Big [ \eta^{2n} \sup_{\substack{
i - j {\thinspace odd} \\
k - l {\thinspace odd} \\
p - q {\thinspace even}
}} \{d(\Xi^i\vartheta,\Xi^j\vartheta), d(\Xi^k(\Xi \omega),\Xi^l(\Xi \omega)), d(\Xi^p\vartheta,\Xi^q (\Xi\omega))\} +(1-\eta^{2n})dist(\Omega,\Delta) \Big]
 \end{align*}
which implies that $d(\Xi\varsigma^*,\Xi\Omega)=dist(\Omega,\Delta)$. Also, $d(\Omega,\Xi\Omega)= dist(\Omega,\Delta)$. Using Lemma \ref{lemma 2.3}, we infer that $d(\Xi\varsigma^*,\Omega)=0$ which gives $\Xi\varsigma^*=\Omega$. 
 \end{proof}
 Suzuki et al. \cite{n22} considered cyclic mappings on $\Omega \cup \Delta$ satisfying  $$d(\Xi\varsigma,\Xi\vartheta) \leq \eta \max \{d(\varsigma,\vartheta), d(\varsigma,\Xi\varsigma), d(\vartheta,\Xi\vartheta)\} + (1- \eta) dist(\Omega,\Delta)$$ for some $\eta \in (0,1)$ and for all $\varsigma \in \Omega$ and $\vartheta \in \Delta$. The existence of best proximity points associated with these mappings \cite[Theorem 2]{n22} has been examined in the context of metric spaces under suitable conditions. We establish   best proximity point results for such contraction mappings in the framework of $CAT_p(0)$ metric spaces.
 \begin{corollary}
 Let $\Omega$ and $\Delta$ be two nonempty, closed and convex subsets of a complete $CAT_{p}(0)$ metric space $(\mathfrak{X},d)$,  with $p \geq 2$. Let $\Xi:\Omega \cup \Delta \rightarrow \Omega \cup \Delta$ be a  mapping such that for all $\varsigma \in \Omega$ and $\vartheta \in \Delta$, we have 
 
(i) $ \Xi(\Omega) \subset \Delta$ and $\Xi(\Delta) \subset \Omega$,
    
(ii) $d(\Xi\varsigma,\Xi\vartheta) \leq \eta 
\max \{d(\varsigma,\vartheta), d(\varsigma,\Xi\varsigma), d(\vartheta,\Xi\vartheta)\} + (1- \eta) dist(\Omega,\Delta)$ for some $\eta \in (0,1)$.\\
Then $\Xi$ has a best proximity point. 
 \end{corollary}
We now establish the results for cyclic orbital contraction mappings in the framework of uniformly convex Banach spaces. We first state the result which is essential for establishing the existence of best proximity point. The proof is omitted as it can be obtained by arguments parallel to those in the proof of Proposition \ref{proposition 2.4}.
\begin{proposition}\label{proposition 2.5}
Let $\Omega$ and $\Delta$ be nonempty subsets of a uniformly convex Banach space $\mathfrak{X}$. Let $\Xi:\Omega \cup \Delta \rightarrow \Omega \cup \Delta$ be a cyclic orbital contraction mapping. Suppose that  $\varsigma_0 \in \Omega$. Define $\varsigma_{n+1}=\Xi\varsigma_{n}$ for each $ n \in \mathbb{N} \cup \{0\}$. Then $\Vert \varsigma_n-\varsigma_{n+1}\Vert \rightarrow dist(\Omega,\Delta)$. 
\end{proposition}
 Using \cite[Lemma 3.7]{n6}, \cite[Lemma 3.8]{n6} and Proposition \ref{proposition 2.5} and proceeding on the similar lines as in Theorem \ref{theorem 2.6}, we get the following result:
 \begin{theorem}\label{theorem 2.7}
    Let $\Omega$ and $\Delta$ be nonempty, closed and convex subsets of a uniformly convex Banach space $\mathfrak{X}$. Let $\Xi:\Omega \cup \Delta \rightarrow \Omega \cup \Delta$ be a cyclic orbital contraction mapping. Then
    
    (i) $\Xi$ has a best proximity point $\varsigma^* \in \Omega$,
    
    (ii) $\varsigma^*$ is the unique fixed point of $\Xi^2$,
    
    (iii) $\{\Xi^{2n}\varsigma\}$ converges to $\varsigma^*$ for every $\varsigma \in \Omega$,
    
    (iv) $\Xi\varsigma^*$ is a best proximity point of $\Xi$ in $\Delta$ and $\{\Xi^{2n}\vartheta \}  $ converges to $\Xi\varsigma^*$ for every $\vartheta \in B$.
\end{theorem}
 \begin{corollary}
 Let $\Omega$ and $\Delta$ be two nonempty, closed and convex subsets of a uniformly convex space $\mathfrak{X}$. Let $\Xi:\Omega \cup \Delta \rightarrow \Omega \cup \Delta$ be a  mapping such that for all $\varsigma \in \Omega$ and $\vartheta \in \Delta$, we have 
 
(i) $ \Xi(\Omega) \subset \Delta$ and $\Xi(\Delta) \subset \Omega$,
    
(ii) $\Vert \Xi\varsigma-\Xi\vartheta\Vert \leq \eta 
\max \{\Vert \varsigma-\vartheta \Vert, \Vert \varsigma-\Xi\varsigma \Vert, \Vert \vartheta-\Xi\vartheta \Vert\} + (1- \eta) dist(\Omega,\Delta)$ for some $\eta \in (0,1)$.\\
Then $\Xi$ has a best proximity point. 
 \end{corollary}
An example illustrating Theorem \ref{theorem 2.7} is given below. 
\begin{example}
\emph{In Example \ref{example2.2}, we see that  $\Omega$ and $\Delta$ are nonempty, closed and convex subsets of a uniformly convex Banach space $\mathfrak{X}=\mathbb{R}^2$. The mapping $\Xi$ is a cyclic orbital contraction mapping. Hence, applying  Theorem \ref{theorem 2.7}, we conclude that there exists $(-1,0) \in \Omega$ satisfying $\Vert (-1,0)-\Xi(-1,0)\Vert=\Vert(-1,0)-(1,0)\Vert=dist(\Omega,\Delta)$. Also, there exists $(1,0) \in \Delta$ satisfying $\Vert (1,0)-\Xi(1,0)\Vert=\Vert(1,0)-(-1,0)\Vert=dist(\Omega,\Delta)$.}
\end{example}
\section*{Acknowledgements}
The first author is grateful to the Government of India for providing financial support in the form of Institutional Fellowship.


\begin{thebibliography}{99}
\bibitem{n15} M. Bessenyei, The contraction principle in extended context, Publ. Math. Debrecen {\bf 89} (2016), no.~3, 287--295.

\bibitem{n18} M.~R. Bridson and A. Haefliger, Metric spaces of non-positive curvature. Springer-Verlag, Berlin (1999).

\bibitem{n16} A. Digar, R. Esp\'inola-Garc\'ia and G.~S.~R. Kosuru, A characterization of weak proximal normal structure and best proximity pairs, Rev. R. Acad. Cienc. Exactas F\'is. Nat. Ser. A Mat. RACSAM {\bf 116} (2022), no.~2, Paper No. 80, 7 pp.

\bibitem{n11} A.~A. Eldred, W.~A. Kirk and P. Veeramani, Proximal normal structure and relatively nonexpansive mappings, Studia Math. {\bf 171} (2005), no.~3, 283--293.  

\bibitem{n6}  A.~A. Eldred and P. Veeramani, Existence and convergence of best proximity points, J. Math. Anal. Appl. {\bf 323} (2006), no.~2, 1001--1006.

\bibitem{n7} R. Esp\'inola-Garc\'ia and A. Fern\'andez-Le\'on, On best proximity points in metric and Banach spaces, Canad. J. Math. {\bf 63} (2011), no.~3, 533--550.

\bibitem{n1} K. Fan, Extensions of two fixed point theorems of F.E. Browder, Math. Z., 112(1969), 234--240.

\bibitem{n17} M. Gabeleh, Remarks on the paper ``A characterization of weak proximal normal structure and best proximity pairs'', Rev. R. Acad. Cienc. Exactas F\'is. Nat. Ser. A Mat. RACSAM {\bf 118} (2024), no.~1, Paper No. 28, 9 pp.

\bibitem{n8} M. Gabeleh and C. Vetro, A note on best proximity point theory using proximal contractions, J. Fixed Point Theory Appl. {\bf 20} (2018), no.~4, Paper No. 149, 11 pp.

\bibitem{n21} A. Safari-Hafshejani, A. Amini-Harandi and M. Fakhar, Best proximity points and fixed points results for noncyclic and cyclic Fisher quasi-contractions, Numer. Funct. Anal. Optim. {\bf 40} (2019), no.~5, 603--619.

\bibitem{n13} M. Heged\H us and T. Szil\'agyi, Equivalent conditions and a new fixed point theorem in the theory of contractive type mappings, Math. Japon. {\bf 25} (1980), no.~1, 147--157.

\bibitem{n19} M.~A. Khamsi and S.~A. Shukri, Generalized CAT(0) spaces, Bull. Belg. Math. Soc. Simon Stevin {\bf 24} (2017), no.~3, 417--426.

\bibitem{n12} W.~A. Kirk, S. Reich and P. Veeramani, Proximinal retracts and best proximity pair theorems, Numer. Funct. Anal. Optim. {\bf 24} (2003), no.~7-8, 851--862.

\bibitem{n9} C. Mongkolkeha, P. Dechboon, S. Salisu and K.  Khammahawong, Coupled best proximity points for cyclic Kannan and Chatterjea contractions in ${\rm CAT}_{\rm p}(0)$ spaces, Carpathian J. Math. {\bf 40} (2024), no.~2, 343--361.

\bibitem{n3} J.B. Prolla, Fixed point theorems for set-valued mappings and existence of best approximants, Numer. Funct. Anal. Optim., 5(1982/83), no.~4, 449--455.

\bibitem{n10}  V.~S. Raj and P. Veeramani, Best proximity pair theorems for relatively nonexpansive mappings, Appl. Gen. Topol. {\bf 10} (2009), no.~1, 21--28.

\bibitem{n2} S. Reich, Approximate selections, best approximations, fixed points and invariant sets, J. Math. Anal. Appl., 62(1978), no.~1, 104--113.

\bibitem{n4} V.M. Sehgal and S.P. Singh, A generalization to multifunctions of Fan's best approximation theorem, Proc. Amer. Math. Soc., 102(1988), no.~3, 534--537.

\bibitem{n5} V.M. Sehgal and S.P. Singh, A theorem on best approximations, Numer. Funct. Anal. Optim., 10(1989), no.~1-2, 181--184.

\bibitem{n20} S.~A. Shukri, Existence and convergence of best proximity points in $\rm{CAT_p}(0)$ spaces, J. Fixed Point Theory Appl. {\bf 22} (2020), no.~2, Paper No. 48, 10 pp.

\bibitem{n22} T. Suzuki, M. Kikkawa and C. Vetro, The existence of best proximity points in metric spaces with the property UC, Nonlinear Anal. {\bf 71} (2009), no.~7-8, 2918--2926.

\bibitem{n14} W.~L. Walter, Remarks on a paper by F. Browder about contraction: ``Remarks on fixed point theorems of contractive type'', Nonlinear Anal. {\bf 5} (1981), no.~1, 21--25.
 \end{thebibliography}
\end{document}